\documentclass[oneside,english]{amsart}
\usepackage[T1]{fontenc}
\usepackage[latin9]{inputenc}
\usepackage{geometry}
\usepackage{color}
\usepackage{babel}
\usepackage{refstyle}
\usepackage{float}
\usepackage{units}
\usepackage{amsthm}
\usepackage{amssymb}
\usepackage{graphicx}
\usepackage[all]{xy}
\usepackage[unicode=true,pdfusetitle,
 bookmarks=true,bookmarksnumbered=false,bookmarksopen=false,
 breaklinks=false,pdfborder={0 0 0},pdfborderstyle={},backref=false,colorlinks=true]
 {hyperref}
\hypersetup{
 citecolor=blue}

\makeatletter

\AtBeginDocument{\providecommand\secref[1]{\ref{sec:#1}}}
\AtBeginDocument{\providecommand\subsecref[1]{\ref{subsec:#1}}}
\AtBeginDocument{\providecommand\figref[1]{\ref{fig:#1}}}
\AtBeginDocument{\providecommand\defref[1]{\ref{def:#1}}}
\AtBeginDocument{\providecommand\claimref[1]{\ref{claim:#1}}}
\AtBeginDocument{\providecommand\thmref[1]{\ref{thm:#1}}}
\AtBeginDocument{\providecommand\lemref[1]{\ref{lem:#1}}}
\AtBeginDocument{\providecommand\corref[1]{\ref{cor:#1}}}
\RS@ifundefined{subsecref}
  {\newref{subsec}{name = \RSsectxt}}
  {}
\RS@ifundefined{thmref}
  {\def\RSthmtxt{theorem~}\newref{thm}{name = \RSthmtxt}}
  {}
\RS@ifundefined{lemref}
  {\def\RSlemtxt{lemma~}\newref{lem}{name = \RSlemtxt}}
  {}

\numberwithin{equation}{section}
\numberwithin{figure}{section}
\theoremstyle{plain}
\newtheorem{thm}{\protect\theoremname}[section]
\theoremstyle{definition}
\newtheorem{defn}[thm]{\protect\definitionname}
\theoremstyle{definition}
\newtheorem{example}[thm]{\protect\examplename}
\theoremstyle{remark}
\newtheorem{rem}[thm]{\protect\remarkname}
\theoremstyle{plain}
\newtheorem{cor}[thm]{\protect\corollaryname}
\theoremstyle{remark}
\newtheorem{claim}[thm]{\protect\claimname}
\theoremstyle{plain}
\newtheorem{lem}[thm]{\protect\lemmaname}

\@ifundefined{date}{}{\date{}}
\hypersetup{pdfstartview=}
\hypersetup{citecolor=blue}

\makeatother

\providecommand{\claimname}{Claim}
\providecommand{\corollaryname}{Corollary}
\providecommand{\definitionname}{Definition}
\providecommand{\examplename}{Example}
\providecommand{\lemmaname}{Lemma}
\providecommand{\remarkname}{Remark}
\providecommand{\theoremname}{Theorem}

\begin{document}
\global\long\def\norm#1{\left\Vert #1\right\Vert }%
\global\long\def\AA{\mathbb{A}}%
\global\long\def\QQ{\mathbb{Q}}%
\global\long\def\PP{\mathbb{P}}%
\global\long\def\CC{\mathbb{C}}%
\global\long\def\HH{\mathbb{H}}%
\global\long\def\ZZ{\mathbb{Z}}%
\global\long\def\NN{\mathbb{N}}%
\global\long\def\KK{\mathbb{K}}%
\global\long\def\LL{\mathbb{L}}%
\global\long\def\RR{\mathbb{R}}%
\global\long\def\FF{\mathbb{F}}%
\global\long\def\oo{\mathcal{O}}%
\global\long\def\aa{\mathcal{A}}%
\global\long\def\bb{\mathcal{B}}%
\global\long\def\limfi#1#2{{\displaystyle \lim_{#1\to#2}}}%
\global\long\def\pp{\mathcal{P}}%
\global\long\def\qq{\mathcal{Q}}%
\global\long\def\da{\mathrm{da}}%
\global\long\def\dt{\mathrm{dt}}%
\global\long\def\dg{\mathrm{dg}}%
\global\long\def\ds{\mathrm{ds}}%
\global\long\def\dm{\mathrm{dm}}%
\global\long\def\dmu{\mathrm{d\mu}}%
\global\long\def\dx{\mathrm{dx}}%
\global\long\def\dy{\mathrm{dy}}%
\global\long\def\dz{\mathrm{dz}}%
\global\long\def\dnu{\mathrm{d\nu}}%
\global\long\def\flr#1{\left\lfloor #1\right\rfloor }%
\global\long\def\cg{\mathfrak{g}}%
\global\long\def\nuga{\nu_{\mathrm{Gauss}}}%
\global\long\def\diag#1{\mathrm{diag}\left(#1\right)}%
\global\long\def\bR{\mathbb{R}}%
\global\long\def\Ga{\Gamma}%
\global\long\def\PGL{\mathrm{PGL}}%
\global\long\def\GL{\mathrm{GL}}%
\global\long\def\SL{\mathrm{SL}}%
\global\long\def\SO{\mathrm{SO}}%
\global\long\def\sl{\mathrm{sl}}%
\global\long\def\mb#1{\mathrm{#1}}%
\global\long\def\wstar{\overset{w^{*}}{\longrightarrow}}%
\global\long\def\vphi{\varphi}%
\global\long\def\av#1{\left|#1\right|}%
\global\long\def\inv#1{\left(\mathbb{Z}/#1\mathbb{Z}\right)^{\times}}%
\global\long\def\cH{\mathcal{H}}%
\global\long\def\cM{\mathcal{M}}%
\global\long\def\bZ{\mathbb{Z}}%
\global\long\def\bA{\mathbb{A}}%
\global\long\def\bQ{\mathbb{Q}}%
\global\long\def\bP{\mathbb{P}}%
\global\long\def\eps{\epsilon}%
\global\long\def\on#1{\mathrm{#1}}%
\global\long\def\nuga{\nu_{\mathrm{Gauss}}}%
\global\long\def\idealeq{\trianglelefteq}%
\global\long\def\fidealeq{\underset{f.i.}{\trianglelefteq}}%
\global\long\def\set#1{\left\{  #1\right\}  }%
\global\long\def\smallmat#1{\begin{smallmatrix}#1\end{smallmatrix}}%
\global\long\def\len{\mathrm{len}}%
\global\long\def\mob{M{\"o}bius}%

\title{Local to global property in free groups}
\author{Ofir David}
\begin{abstract}
The local to global property for an equation $\psi$ over a group
$G$ asks to show that $\psi$ is solvable in $G$ if and only if
it is solvable in every finite quotient of $G$. In this paper we
focus show that in order to prove this local to global property for
free groups $G=F_{k}$, it is enough to prove for $k\leq$ the number
of parameters in $\psi$. In particular we use it to show that the
local to global property holds for $m$-powers in free groups.
\end{abstract}

\email{eofirdavid@gmail.com}
\maketitle

\section{Introduction}

An interesting concept that appears in many parts of mathematics is
the local to global principle. One of the most well known examples
is the problem of finding integer solutions to $x^{2}+y^{2}=p$ where
$p\in\ZZ$ is some prime. If there was an integer solution, then there
is a solution mod $n$ for every $n$. In particular we have a solution
mod $4$, and since the squares mod 4 are $0$ and $1$, this implies
that $p\equiv_{4}0,1,2$, so there is no solution if $p\equiv_{4}3$.
Similar results hold for other equations over $\ZZ$.

A natural question is if the converse holds as well - if there is
a family of solutions mod $n$ for all $n$, does it implies a solution
over $\ZZ$? In our example above, if $p\not\equiv_{4}3$ is a prime,
then it is well known that there is an integer solution to $x^{2}+y^{2}=p$
(so in this case, it is enough to find a solution mod $4$).\\

In this paper we study this local to global principle for groups with
respect to their finite quotients. Namely, if $G$ is a group and
$\psi$ is an equation over $G$ (defined below), is it true that
$\psi$ has a solution in $G$ if and only if it has a solution in
every finite quotient of $G$? For example, in the additive group
$\ZZ$, given $m,n\in\NN$, we may ask whether $n=\sum_{1}^{m}x=mx$
has an integer solution $x\in\ZZ$, and this can be easily seen to
be equivalent to whether $n=mx$ has a solution in the finite quotient
$\nicefrac{\ZZ}{m\ZZ}$.
\begin{defn}
Denote by $F_{n}=\left\langle x_{1},...,x_{n}\right\rangle $ the
free group on $n$ letters, and for a group $G$ denote by $F_{n}*G$
the free product. If $\varphi:F_{n}\to G$ is any homomorphism, we
will also denote by $\varphi:F_{n}*G\to G$ its natural extension
which is the identity on $G$.

An \emph{equation $\psi$ over $G$ }is simply an element $\psi\in F_{n}*G$.
We say that \emph{$\psi$ has a solution over $G$ }if there is some
homomorphism $\varphi\in Hom\left(F_{n},G\right)$ such that $\varphi\left(\psi\right)=e_{G}$. 

If $\pi:G\to H$ is a surjective homomorphism, then we say that $\psi$
has a solution in $H$ (with respect to $\pi$) if there is some $\varphi\in Hom\left(F_{n},G\right)$
such that $\pi\left(\varphi\left(\psi\right)\right)=e_{H}$.
\end{defn}

\newpage{}
\begin{example}
\begin{enumerate}
\item In the example before the definition we had the equation $\sum_{1}^{m}x-n$
where $F_{1}=\left\langle x\right\rangle $ as an additive group.
We can write the equation as an additive equation because $\ZZ$ is
abelian. More generally, for $g\in G$ we can consider the equation
$x^{m}g^{-1}$, which is solvable in $G$ exactly when $g$ is an
$m$-power.
\item If $G$ is any group and $g\in G$, we consider the equation $\psi\left(x,y\right)=\left[x,y\right]g^{-1}=xyx^{-1}y^{-1}g^{-1}$.
This equation has a solution over $G$ exactly if $g$ is a commutator.
If $G$ is any abelian group and $e\neq g\in G$, than $\psi$ doesn't
have a solution in $G$ - this is because every commutator in an abelian
group is trivial. However, if $\pi:G\to H$ is a projection such that
$\pi\left(g\right)=e_{H}$, then $\psi$ has a solution over $H$
because, for example, $\pi\left(\psi\left(e,e\right)\right)=\pi\left(g\right)=e_{H}$.
\item If $g,h\in G$, then the equation $xgx^{-1}h^{-1}\in\left\langle x\right\rangle *G$
is solvable if and only if $g$ and $h$ are conjugate to one another.
\end{enumerate}
\end{example}

\begin{rem}
We will mostly be interested in equations as in $(1)$ and $(2)$
above with the form $w=g$ with $w\in F_{n},\;g\in G$ (i.e. $x^{m}=g$
and $\left[x,y\right]=g$). However, some of the result here are true
in the more general definition and we prove it for them.
\end{rem}

\begin{defn}
Let $\psi$ be an equation over $G$. We say that $\psi$ has the
local to global property, if $\psi$ is solvable over $G$ (global
solution) if and only if it is solvable over $\nicefrac{G}{K}$ for
any $K\fidealeq G$ (local solutions).
\end{defn}

For a general group, there is no reason that local solutions will
imply a global solution. Indeed, there might be some $e\neq g\in G$
which is trivial in every finite quotient, so we can't even distinguish
it from $e$. But even if this is not the case and $G$ is residually
finite, then it might still not be enough, and we will need to work
with its profinite completion $\hat{G}$ (see Definitions \ref{def:residually_finite},
\ref{def:profinite_completion}). This completion is a topological
completion, and this topological approach let us use notations and
results from topology which contribute in both understanding better
these local to global problems and solve them.

In this paper we are mainly interested in the local to global principle
for free groups. In the two examples mentioned above this principle
was shown to hold, first in \cite{thompson_power_1997} by Thompson
(though the proof is by Lubotzky) for powers (namely $x^{m}=g$) and
then in \cite{khelif_finite_2004} by Khelif for commutators (namely
$\left[x,y\right]=g$). In this paper we give another proof to a more
generalized form of a reduction step appearing in \cite{khelif_finite_2004}
for equations over free groups. 
\begin{thm}
Let $w\in F_{n}$ be a word on $n$ variables. Then $wg^{-1}$ has
the local to global property for any $g\in G=F_{k}$ and any $k\geq0$
, if this claim is true for $0\leq k\leq n$.
\end{thm}

Actually, the theorem above will be slightly stronger - it is enough
to prove the local to global property if we further assume that for
any finite quotient $\nicefrac{G}{K}$, the solution $h_{1},...,h_{k}$
to the equation $w=g$ also generate $\nicefrac{G}{K}$.

This reduction together with the discussion above about the equation
$n=mx$ over $\ZZ$ produce another, very simple, proof for the local
to global principle for the equation $x^{m}=g$ over free groups.
\begin{thm}
\label{thm:main_theorem}Let $G=F_{k}$ be the free group on $k$
variables. Then for any $g\in G$, we can write $g=h^{m}$ for some
$h\in G$, if and only if in any finite quotient $\nicefrac{G}{K}$
of $G$ we can find $h_{K}\in G$ such that $gK=h_{K}^{m}K$.
\end{thm}

\newpage{}

In \secref{Profinite-topology} we will begin by recalling the main
definitions and results for profinite groups and their completions,
and in particular how to interpret solution of equations in a topological
language. 

After giving a sketch of the idea of the reduction step in \subsecref{The-reduction-idea},
we recall the definition of Stallings graphs in \subsecref{Stallings-graphs}
which is a very useful topological tool to study free groups, and
finally in \subsecref{Proof-main} we give the proof of the reduction
step, and its application to $m$-powers.

For the reader convenience, we added \secref{Some-profinite-results}
where we give proofs for some of the well known results about profinite
groups that we used in the reduction step.

\subsection{Acknowledgments}

I would like to thank Michael Chapman for introducing me to these
local to global questions for free groups, and asking me to give a
lecture about them in a Stallings' graphs seminar that he and Noam
Kolodner organized at the Hebrew university. Here is hoping that next
time I will not have to read a paper that was never meant to be read.

The research leading to these results has received funding from the
European Research Council under the European Union Seventh Framework
Programme (FP/2007-2013) / ERC Grant Agreement n. 335989.

\section{\label{sec:Profinite-topology}Profinite topology and completion
- motivation}

Let us fix some group $G$ and an equation $\psi\in F_{n}*G$. As
we mentioned before, if $\psi$ has a solution in $G$, then it has
a solution in every quotient (and in particular finite quotients)
of $G$. We now try to understand what we need for the inverse direction
to hold as well. 
\begin{defn}
Let $F_{n}$ be the free group generated by $x_{1},...,x_{n}$. Given
$\bar{g}=\left(g_{1},...,g_{n}\right)\in G^{n}$ we let $\varphi_{\bar{g}}:F_{n}\to G$
be the homomorphism defined by $\varphi_{\bar{g}}\left(x_{i}\right)=g_{i}$
Note that $\varphi_{\bar{g}}$ depends on our choice of the basis
$x_{1},...,x_{n}$ for $F_{n}$, but once this basis is chosen, every
homomorphism can be written uniquely as $\varphi_{\bar{g}}$, so we
can identify $Hom\left(F_{n},G\right)$ with $G^{n}$. We will also
use the same notation for the extension $\varphi_{\bar{g}}:F_{n}*G\to G$
which is the identity on $G$. Finally, for $\psi\in F_{n}*G$ we
will write
\[
\psi\left(\bar{g}\right):=\varphi_{\bar{g}}\left(\psi\right).
\]
\end{defn}

\begin{example}
Fix some group $G$, and element $g\in G$ and let $\psi=\left[x,y\right]g^{-1}\in F_{2}*G$.
If $g_{1},g_{2}\in G$, then our two notations will produce 
\begin{align*}
\varphi_{\left(g_{1},g_{2}\right)}\left(\psi\right) & =\psi\left(g_{1},g_{2}\right)=\left[g_{1},g_{2}\right]g^{-1}.
\end{align*}
\end{example}

We can now talk about the set of solutions in $G^{n}$ for this equation.
\begin{defn}
Given any normal subgroup $K\idealeq G$ we write 
\[
\Omega_{\psi,K}=\left\{ \bar{g}\in G^{n}\;\mid\;\psi\left(\bar{g}\right)\in K\right\} .
\]
In other words, these are all the tuples which solve the equation
$\psi$, after projecting to $\nicefrac{G}{K}$. We will also write
$\Omega_{\psi}:=\Omega_{\psi,\left\{ e\right\} }$ which are the actual
solutions for $\psi$ in $G$.
\end{defn}

We should think of $\Omega_{\psi,K}$ as approximate solutions - they
might not solve the original equation, but up to an ``error'' in
$K$ they do. Clearly, if $K_{1}\leq K_{2}$ are normal in $G$, then
$\Omega_{\psi,K_{1}}\subseteq\Omega_{\psi,K_{2}}$, so we have better
and better solutions as we decrease $K$. 

If we run over the finite index normal subgroups $K$, then we get
that 
\[
\Omega_{\psi}\subseteq\bigcap_{K\underset{f.i}{\idealeq}G}\Omega_{\psi,K}.
\]

\newpage This is equivalent to our statement from before that a global
solution in $G$ implies a local solution in every finite quotient.
Assume now that we have a solution in every finite quotient, namely
the $\Omega_{\psi,K}$ are not empty, and we want to find a global
solution, namely $\Omega_{\psi}\neq\emptyset$. 

This will follow if we can prove the next two properties:
\begin{enumerate}
\item There is an equality $\Omega_{\psi}=\bigcap_{K\underset{f.i}{\idealeq}G}\Omega_{\psi,K}$,
and
\item The intersection $\bigcap_{K\underset{f.i}{\idealeq}G}\Omega_{\psi,K}$
is nonempty.
\end{enumerate}
For the first part, suppose that $\bar{g}\in\bigcap_{K\underset{f.i}{\idealeq}G}\Omega_{\psi,K}$,
or equivalently the value $\psi\left(\bar{g}\right)\in\bigcap K$.
If we knew that the intersection of all finite index normal subgroups
is trivial, then it means that $\psi\left(\bar{g}\right)=e$, and
therefore $\bar{g}$ is a solution in $G$ for $\psi$. This leads
to the following definition:
\begin{defn}
\label{def:residually_finite}A group $G$ is called residually finite,
if for any $e\neq g\in G$ there is a projection $\pi:G\to H$ to
some finite group $H$ such that $\pi\left(g\right)\neq e$. Equivalently
the intersection of all finite index subgroups (resp. f.i. normal
subgroups) is trivial.
\end{defn}

\begin{example}
\begin{enumerate}
\item The group $\ZZ$ is residually finite. If $n\neq0$, then its reduction
modulo $2n$ is non trivial.
\item Similarly the group $\SL_{k}\left(\ZZ\right)$ for $k\geq2$ is residually
finite by considering the maps $\SL_{k}\left(\ZZ\right)\to\SL_{k}\left(\nicefrac{\ZZ}{p\ZZ}\right)$
for primes $p$.
\item Subgroups of residually finite groups are also residually finite.
Since the free group $F_{2}\cong\left\langle \left(\begin{array}{cc}
1 & 0\\
2 & 1
\end{array}\right),\left(\begin{array}{cc}
1 & 2\\
0 & 1
\end{array}\right)\right\rangle \leq\SL_{2}\left(\ZZ\right)$ we get that it is residually finite. Moreover, every free group $F_{k}$
can be embedded in $F_{2}$, so all the free groups are residually
finite.
\end{enumerate}
\end{example}

If $G$ is residually finite, then the first condition above holds.
For the second condition, note that if $K_{j}\fidealeq G,\;k=1,...,k$,
then so is $\bigcap_{1}^{k}K_{j}\fidealeq G$. It follows that $\Omega_{\psi,\bigcap_{1}^{k}K_{j}}\subseteq\bigcap_{1}^{k}\Omega_{\psi,K_{j}}$,
and under our assumption that there are always local solutions for
every $K\fidealeq G$, we conclude that any finite intersection $\bigcap_{1}^{k}\Omega_{\psi,K_{j}}$
is nonempty. If we can give a topology to $G$ where the $\Omega_{\psi,K}$
are compact, then the finite intersection property for compact sets
implies that $\bigcap_{K\underset{f.i}{\idealeq}G}\Omega_{\psi,K}\neq\emptyset$
which is the second condition above.

In general, this might not be the case, however there is a natural
topology on $G$, called the profinite topology, and then completing
$G$ with respect to this topology will have this property.
\begin{defn}
\label{def:profinite_completion}Let $G$ be a residually finite group,
and let $\left\{ \left(\pi_{i},H_{i}\right)\;\mid\;i\in I\right\} $
be all the pairs such that $\pi_{i}:G\to H_{i}$ is a projection to
a finite group $H_{i}$. Let $H_{i}$ have the discrete topology and
$\prod_{i\in I}H_{i}$ have the product topology and identify $G$
as a subgroup via $g\mapsto\left(\pi_{i}\left(g\right)\right)_{i\in I}$
(this is injective because $G$ is residually finite). Then the induced
topology on $G$ is called the profinite topology, and the closure
of $G$ in $\prod_{i\in I}H_{i}$ , denoted by $\hat{G}$, is called
the profinite completion of $G$. For each $i\in I$, we will denote
by $\hat{\pi}_{i}:\hat{G}\to H_{i}$ the projection of $\hat{G}$
to $H_{i}$ which is an extension of $\pi_{i}$.
\end{defn}

\newpage{}

In the profinite topology, the cosets of finite index subgroups form
a basis for the topology. Thus, we should think of two elements as
``close'' if they are mapped to the same element under the projection
$G\to\nicefrac{G}{K}$ where $\left|\nicefrac{G}{K}\right|$ is large.
Alternatively, the profinite topology is the weakest topology such
that all the projections to the (discrete) finite groups are continuous.
For more details about residually finite groups and their completion,
the reader is referred to \cite{ceccherini-silberstein_cellular_2010,wilson_profinite_1998,ribes_profinite_2010}. 

For our discussion, one of the main results that we need is the following:
\begin{thm}
Let $G$ be a residually finite group and $\hat{G}$ its profinite
completion. Then $\hat{G}$ is a totally disconnected, Hausdorff and
compact topological group which is also residually finite.

\end{thm}

\begin{proof}
The finite groups with discrete topologies are totally disconnected,
Hausdorff and compact groups. It is easily seen that the product of
such sets is totally disconnected and Hausdorff, and by Tychonoff's
theorem it is also compact. Since by definition $\hat{G}$ is closed
in a product of finite groups, we conclude that it is also totally
disconnected, Hausdorff and compact. Finally, product of finite group
is always residually finite, and therefore $\hat{G}$ is residually
finite as a subgroup of a residually finite group.
\end{proof}
We already see the usefulness of this topological approach, and in
particular Tychonoff's theorem above shows that $\hat{G}$ is compact.

The profinite topology is defined in such a way that if $K\fidealeq G$,
then the natural map $G\to\nicefrac{G}{K}$ is continuous, so the
induced map $\psi_{K}:G^{n}\overset{\psi}{\to}G\to\nicefrac{G}{K}$
is also continuous, and therefore $\Omega_{\psi,K}=\psi_{K}^{-1}\left(eK\right)$
is closed in $G^{n}$. Furthermore, their closures $\hat{\Omega}_{\psi,K}$
in $\hat{G}$ are compact so we can use their finite intersection
property in $\hat{G}$.
\begin{cor}
Let $G$ be a residually finite group, and let $\psi$ be an equation
over $G$. Then $\psi$ has a solution in $\hat{G}$, if and only
if it has a solution in every finite quotient of $G$.
\end{cor}

\begin{proof}
The first part, namely local solutions imply a solution in $\hat{G}$
was presented above, and it is left to the reader to fill in the details.
For the other direction, if $K\fidealeq G$, then the map $\pi:G\to\nicefrac{G}{K}$
can be extended to $\hat{\pi}:\hat{G}\to\nicefrac{G}{K}$, so a solution
in $\hat{G}$ implies a local solution for every such $K$.
\end{proof}
Given an equation $\psi$ over $G$, it is also an equation over $\hat{G}$,
so by the corollary above, the set of solutions there $\hat{\Omega}_{\psi}$
is not empty exactly if there is a solution in every finite quotient
of $G$. This let us talk about the local to global principal in a
more ``compact'' way. Namely, the local to global principal holds
exactly if a solution in $\hat{G}$ implies a solution in $G$.

\newpage{}

\section{The reduction}

\subsection{\label{subsec:The-reduction-idea}The reduction idea}

Now that we have our new topological notation, let us sketch the main
idea of the reduction step, where as example we consider the equation
$x^{m}=g$.
\begin{enumerate}
\item Fix some $g\in G=F_{k}$ and assume that $g$ is an $m$-power in
every finite quotient of $F_{k}$ or equivalently there is some $h\in\hat{G}$
such that $h^{m}=g$. Letting $H=\overline{\left\langle h\right\rangle }\leq\hat{G}$
, assume first that there is some $g\in H_{0}\leq F_{k}$ finitely
generated such that $\widehat{H}_{0}\cong\overline{H_{0}}=H$, namely
\[
\xymatrix{ & \overline{H}_{0}=\overline{\left\langle h\right\rangle }\ar[r] & \hat{F}_{k}\\
\left\langle g\right\rangle \ar[r] & H_{0}\ar[r]\ar[u] & F_{k}.\ar[u]
}
\]
The group $H_{0}$ is free as a subgroup of a free group, and therefore
$H_{0}\cong F_{k'}$ for some $k'$ (it is also finitely generated).
On the other hand, its completion $\hat{H}_{0}\cong\overline{\left\langle h\right\rangle }$
is generated topologically by one element $h$, so we must have that
$k'=1$. In other words, this reduces the problem to $H_{0}\cong\ZZ$
and $\hat{H}_{0}\cong\hat{\ZZ}$ where we already know the local to
global principle for $m$-powers (or additively - multiples of $m$).
\item However, the conditions above are not true in general. In order to
fix this problem, let $H_{0}$ be a finitely generated subgroup such
that $g\in H_{0}$ and $\overline{H_{0}}\leq H$. Suppose that we
can find a surjective continuous map $\pi:H\to\overline{H_{0}}$ which
fixes $\overline{H_{0}}$, so we have the diagram
\[
\xymatrix{ & \overline{H}_{0}\ar[r]_{\iota} & H\ar[r]\ar@(ul,ur)[l]_{\pi} & \hat{F}_{k}\\
\left\langle g\right\rangle \ar[r] & H_{0}\ar[rr]\ar[u] &  & F_{k},\ar[u]
}
\]
where $\pi\circ\iota=Id_{\overline{N}}$. In this case we get that
(1) $\overline{H_{0}}$ will be generated topologically by $\pi\left(h\right)$
and $(2)$ we have that $\pi\left(h\right)^{m}=\pi\left(h^{m}\right)=\pi\left(g\right)=g$
because $g\in\overline{H_{0}}$ is fixed by $\pi$. We can now use
the trick from above for $\left\langle g\right\rangle \leq H_{0}\leq\overline{H_{0}}\leq\hat{G}$.\\
As we shall see, this result will be more general, and not only for
the $m$-power equations. The interpretation of this result will be
that if $w\in F_{n}$ is any word (e.g. $w=x^{m},\left[x,y\right]$
etc.) and we want to prove the local to global property for it, namely
that a solution to $w\left(\hat{g}_{1},...,\hat{g}_{n}\right)=g$
in $\hat{G}$ implies a solution in $G$, it is enough to prove this
claim under the further assumption that $\hat{g}_{1},...,\hat{g}_{n}$
generate $\hat{F}_{k}$, so in particular $k\leq n$. In the finite
quotients language, it means that the solution $\hat{g}_{1}K,...,\hat{g}_{n}K$
generate the group $\nicefrac{G}{K}$ for every $K\fidealeq G$. 
\item The two main problems in (2) above are to somehow define a homomorphism
from $H$ to $\overline{H_{0}}$ and moreover it needs to fix $\overline{H_{0}}$.
However, there is one such situation where this is very easy. Suppose
that $H_{0}$ is a free factor $H_{0}\leq_{*}N$ of some group $N\leq F_{k}$,
namely we can write $N=H_{0}*H_{1}$ for some subgroup $H_{1}$. Then
there is a natural projection $\pi:N\to H_{0}$ which fixes $H_{0}$.
We can then extend it to their completion $\overline{\pi}:\overline{N}\to\overline{H}_{0}$,
and if $H\leq\overline{N}$ then the restriction of $\overline{\pi}$
to $H$ will do the job
\[
\xymatrix{ & \overline{H}_{0}\ar[r]_{\iota} & H\ar[r] & \overline{N}\ar[r]\ar@(ul,ur)[ll]_{\pi} & \hat{F}_{k}\\
\left\langle g\right\rangle \ar[r] & H_{0}\ar[rr]\ar[u] &  & N\ar[r]\ar[u] & F_{k}.\ar[u]
}
\]
\end{enumerate}
The main result of the next section will be to show that we can actually
choose such $H_{0}$ and $N$ ``wisely'' so that $H_{0}\leq_{*}N$
and $H\leq\overline{N}$.

\subsection{\label{subsec:Stallings-graphs}Stallings graphs}

The Stallings graphs will be our main tool to understand subgroups
of the free group, and to say when one subgroup is a free factor of
another. Before we define them, consider the following example of
a labeled graph (defined below).

\begin{figure}[H]
\begin{centering}
\includegraphics[scale=0.5]{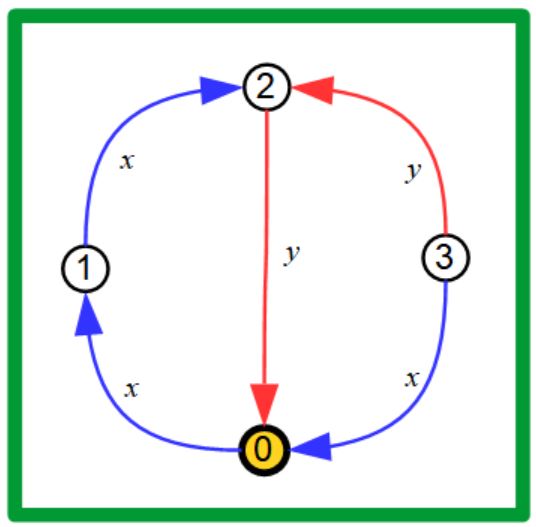}
\par\end{centering}
\caption{\label{fig:first_example} A labeled graph.}
\end{figure}
This graph has two ``main'' cycles corresponding to $x^{2}y$ on
the left and $x^{-1}y^{2}$ on the right, and every other cycle can
be constructed using these two cycles (up to homotopy, i.e. modulo
backtracking). Furthermore, the labeling allows us to think of these
cycles as elements in $F_{2}=\left\langle x,y\right\rangle $, so
that the fundamental group of the graph could be considered as the
subgroup generated by $x^{2}y$ and $x^{-1}y^{2}$. With this example
in mind, we now give the proper definitions to make this argument
more precise.\\

One of the most basic results in algebraic topology is that the fundamental
group of a graph is always a free group. Let us recall some of the
details.
\begin{defn}[Cycle basis]
\label{def:Cycle_basis}Let $\Gamma$ be a connected undirected graph
with a special vertex $v\in V\left(\Gamma\right)$, and let $T\subseteq E\left(\Gamma\right)$
be a spanning tree. For each edge $e:u\to w$ let $C_{e}$ be the
simple cycle going from $v$ to $u$ on the unique path in the tree
$T$, then from $u$ to $w$ via $e$ and finally from $w$ to $v$
via $T$. We denote by $C\left(T\right)=\left\{ e\notin T\;\mid\;C_{e}\right\} $
this collection of cycles. 
\end{defn}

It is not hard to show that any cycle in a connected graph can be
written as a concatenation of cycles in $C\left(T\right)$ and their
inverses as elements in the fundamental group $\pi_{1}\left(\Gamma\right)$
(namely, we are allowed to remove backtracking). More over, it has
a unique such presentation which leads to the following:
\begin{cor}
\label{cor:cycle_basis}Let $\Gamma$ be a graph and $T$ a spanning
tree. Then $C\left(T\right)$ is a basis for $\pi_{1}\left(\Gamma\right)$
which is a free group on $\left|E\left(\Gamma\right)\right|-\left|V\left(\Gamma\right)\right|+1$
elements. 
\end{cor}

\begin{example}
In \figref{first_example} the edge touching the 0 vertex form a spanning
tree, and then $C_{\left(1,2\right)}=0\overset{x}{\to}1\overset{x}{\to}2\overset{y}{\to}0$
and $C_{\left(3,2\right)}=0\overset{x}{\leftarrow}3\overset{y}{\to}2\overset{y}{\to}0$,
so that eventually we will think of the fundamental group as generated
by $x^{2}y$ and $x^{-1}y^{2}$ per our intuition from the start of
this section.
\end{example}

In particular, the corollary above implies that the fundamental group
of the bouquet graph with a single vertex and $n$ self loops is the
free group $F_{n}$. We can label the edges by their corresponding
basis elements $x_{1},...,x_{n}$ of $F_{n}$. Since it is important
in which direction we travel across the edge, we will think of each
edge as two directed edges labeled by $x_{i}$ and $x_{i}^{-1}$ depending
on the image in the fundamental group. For simplicity, we will keep
only the edges with the $x_{i}$ labeling, understanding that we can
also travel in the opposite direction via an $x_{i}^{-1}$ labeled
edge. 

It is well known that a fundamental group of a covering space correspond
to a subgroup of the original space. Using the generalization of the
labeling above we can produce covering using the combinatorics of
labeled graphs.

For the rest of this section we fix a basis $x_{1},...,x_{n}$ of
the free group $F_{n}$. 
\begin{defn}
A labeled graph $\left(\Gamma,v\right)$ is a directed graph $\Gamma$
with a special vertex $v$, where the edges are labeled by $x_{1},...,x_{n}$
(see \figref{labeled_graphs}). A labeled graph morphism $\left(\Gamma_{1},v_{1}\right)\to\left(\Gamma_{2},v_{2}\right)$
between labeled graphs is a morphism of graphs $\Gamma_{1}\to\Gamma_{2}$
which sends $v_{1}$ to $v_{2}$ and preserves the labels. 
\end{defn}

We denote by $\Gamma_{F_{n}}$ the bouquet graph with the $x_{1},...,x_{n}$
labeling. Note that another way to define a labeling on a graph $\Gamma$
is a morphism of directed graphs $\varphi:\Gamma\to\Gamma_{F_{n}}$
where the labeling of an edge $e\in E\left(\Gamma\right)$ is defined
to be the labeling of $\varphi\left(e\right)$. In this way a labeled
graph morphism is just a map which defines a commuting diagram
\[
\xymatrix{\left(\Gamma_{1},v_{1}\right)\ar[r]\ar[rd] & \left(\Gamma_{2},v_{2}\right)\ar[d]\\
 & \Gamma_{F_{n}}.
}
\]
This labeling map $\varphi:\Gamma\to\Gamma_{F_{n}}$ induces a homomorphism
$\hat{\varphi}:\pi_{1}\left(\Gamma,v\right)\to\pi_{1}\left(\Gamma_{F_{n}}\right)=F_{n}$.
Since every path in $\Gamma$ is sent to a cycle in $\Gamma_{F_{n}}$,
we can extend this map to general paths in $\Gamma$.
\begin{defn}
Let $\left(\Gamma,v\right)$ be a graph with a labeling $\varphi:\Gamma\to\Gamma_{F_{n}}$.
Given a path $P$ in $\Gamma$ starting at $v$, define the labeling
$L\left(P\right)$ of the path to be the (cycle) element $\varphi\left(P\right)$
in the fundamental group $\pi_{1}\left(F_{n}\right)$. In other words,
this is just the element in $F_{n}$ created by the labels on the
path.
\end{defn}

In general, for a labeled graph $\varphi:\Gamma\to\Gamma_{F_{n}}$
the function $\hat{\varphi}$ is not injective. However, in the Stallings
graphs case, defined below, it is.
\begin{defn}
A Stallings graph is a labeled graph $\varphi:\left(\Gamma,v\right)\to\Gamma_{F_{n}}$
where $\varphi$ is locally injective, namely for every vertex $u\in V\left(\Gamma\right)$
and every $i=1,...,n$ there is at most one outgoing edge from $u$
and at most one ingoing edge into $u$ labeled by $x_{i}$. We call
the graph a covering graph if $\varphi$ is a local homeomorphism,
or equivalently every vertex has exactly one ingoing and one outgoing
labeled by $x_{i}$ for every $i$.
\end{defn}

\begin{rem}
Given a covering graph, we can remove every edge and vertex which
are not part of a simple cycle so as to not change the fundamental
group. The resulting graph will be a Stallings graph, and conversely,
every Stallings graph can be extended to a covering graph of $\Gamma_{F_{n}}$
without changing the fundamental domain.
\end{rem}

In the Stallings graph case, it is an exercise to show that $\hat{\varphi}$
is injective, and we may consider $\pi_{1}\left(\Gamma,v\right)$
as a subgroup of $F_{n}$. Moreover, we can use \defref{Cycle_basis}
to find a basis for $\pi_{1}\left(\Gamma,v\right)$ as a subgroup
of $F_{n}$.
\begin{example}
In \figref{labeled_graphs} below, in the left most graph, the path
\[
P:=v_{0}\overset{x}{\longrightarrow}v_{1}\overset{y}{\longrightarrow}v_{2}\overset{y}{\longrightarrow}v_{0}\overset{y}{\longleftarrow}v_{0}
\]
is labeled by $L\left(P\right)=xyyy^{-1}=xy$. Similarly, in the second
graph from the right the path 
\[
P:=v_{0}\overset{x}{\longrightarrow}v_{1}\overset{y}{\longrightarrow}v_{2}\overset{x}{\longleftarrow}v_{3}\overset{y}{\longleftarrow}v_{0}
\]
is labeled by $\pi_{1}\left(P\right)=xyx^{-1}y^{-1}=\left[x,y\right]$. 

The images $\hat{\varphi}\left(\Gamma,v\right)$ for the graphs in
this figure from left to right are $\left\langle xy,xy^{2},y\right\rangle =\left\langle x,y\right\rangle ,\;\left\langle y,xyx^{-1}\right\rangle $,
$\left\langle xyx^{-1}y^{-1}\right\rangle $ and $\left\langle x,y\right\rangle $.
Note that the fundamental group of the left most graph is free of
rank $3$ (there are 3 loops in the graph) while the image $\hat{\varphi}\left(\Gamma,v\right)=\left\langle x,y\right\rangle $
is generated by only two element, which in particular indicates that
it is not a Stallings graph.
\end{example}

\begin{figure}[H]
\begin{centering}
\includegraphics[scale=0.6]{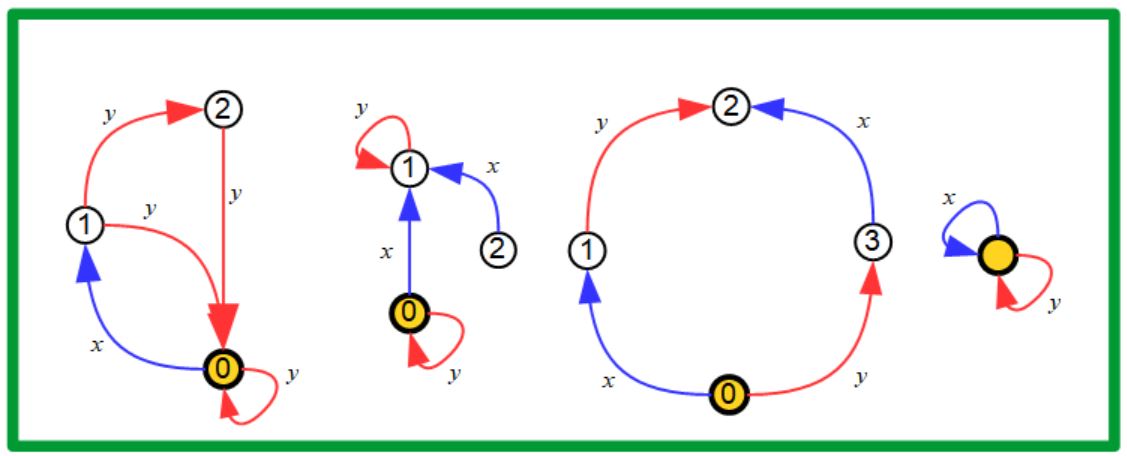}
\par\end{centering}
\caption{\label{fig:labeled_graphs}These are graphs labeled by $x,y$ where
$F_{2}=\left\langle x,y\right\rangle $, and the special vertices
are the yellow ones. The left most graph is not Stallings because
it has two $y$ labeled edges coming out of the same vertex. The rest
are Stallings graphs where the right most graph is $\Gamma_{F_{2}}$.}
\end{figure}
One way to construct Stallings graphs is by starting with a standard
labeled graph, and then taking a suitable quotient. We will use this
construction to build such graphs for subgroups of $F_{n}$.
\begin{defn}
Let $\varphi:\left(\Gamma,v\right)\to\Gamma_{F_{n}}$ be a labeled
graph. Define an equivalence relation on vertices $v_{1}\sim v_{2}$
if there exist paths $P_{1},P_{2}$ leading from $v_{0}$ to $v_{1},v_{2}$
respectively such that $\pi_{1}\left(P_{1}\right)=\pi_{1}\left(P_{2}\right)$.
Define an equivalence relation on the edges $\left(v_{1}\overset{x_{i}}{\to}v_{1}'\right)\sim\left(v_{2}\overset{x_{j}}{\to}v_{2}'\right)$
if $v_{1}\sim v_{2}$ and $x_{i}=x_{j}$ (which implies that $v_{1}'\sim v_{2}'$
also). 
\end{defn}

\newpage{}

It is easy to check that if $\left(\Gamma,v\right)$ is a labeled
graph, then the quotient graph $\left(\nicefrac{\Gamma}{\sim},\left[v\right]\right)$,
where $\left[v\right]$ is the image of $v$, is a Stallings graph.
We will usually also remove any edges and vertices which are not part
of a simple cycle, since these do not change the fundamental group.
\begin{example}
In the left most graph in \figref{labeled_graphs}, the path $v_{0}\overset{x}{\longrightarrow}v_{1}\overset{y}{\longrightarrow}v_{2}$
and the path $v_{0}\overset{x}{\longrightarrow}v_{1}\overset{y}{\longrightarrow}v_{0}$
define the same element in $F_{2}$, so we need to identify the vertices
$v_{0}$ and $v_{2}$. Similarly the paths $v_{0}\overset{y}{\longleftarrow}v_{0}$
and $v_{0}\overset{y}{\longleftarrow}v_{1}$ define the same element
so we need to identify $v_{0}$ and $v_{1}$, so in the end we are
left with the bouquet graph $\Gamma_{F_{2}}$.
\end{example}

\begin{defn}
Let $S\subseteq F_{n}$. For each $s\in S$, let $\left(P_{s},v_{s,0}\right)$
be a cycle graph on a single path $P_{s}$ such that $\pi_{1}\left(P_{s}\right)=s$.
Let $\Gamma$ be the graph $\bigcup_{s\in S}\left(P_{s},v_{s,0}\right)$
where we identify all the $v_{s,0}$ into a single vertex. Denote
by $\Gamma_{S}=\nicefrac{\Gamma}{\sim}$ its quotient.
\end{defn}

\begin{claim}
For any $S\subseteq F_{n}$ we have that $\pi_{1}\left(\Gamma_{S}\right)=\left\langle S\right\rangle $.
\end{claim}

\begin{proof}
Left as an exercise to the reader.
\end{proof}
With our new language of Stallings graphs, we can now prove how a
simple condition of injectivity implies that one subgroup is a free
factor of another subgroup.
\begin{claim}
\label{claim:injective_is_direct_summand}Let $\left(\Gamma,v\right)$
be a Stallings graph and $\left(\Gamma',v\right)$ a labeled subgraph
(which must be Stallings as well). Then $\pi_{1}\left(\Gamma',v\right)$
is a free factor of $\pi_{1}\left(\Gamma,v\right)$.
\end{claim}

\begin{proof}
Let $T'$ be a spanning tree for $\Gamma'$ and extend it to a tree
$T$ of $\Gamma$. Our construction of cycle basis $C\left(T\right)$
will contain the cycle basis $C\left(T'\right)$, so that we can choose
generators for $\pi_{1}\left(\Gamma',v\right)$ which is a subset
of a set of generators for $\pi_{1}\left(\Gamma,v\right)$, which
implies that the first is a free factor of the second.
\end{proof}
\begin{example}
Consider the following Stallings graph:

\begin{figure}[H]
\begin{centering}
\includegraphics[scale=0.6]{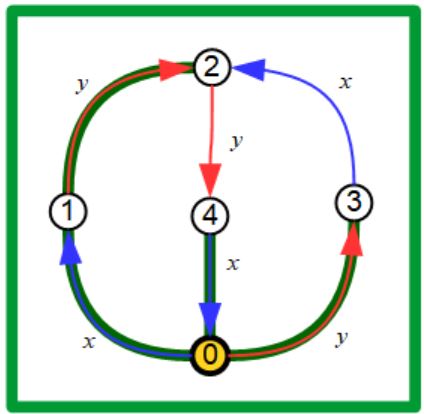}
\par\end{centering}
\caption{\label{fig:stallings_tree}A Stallings graph where the green edges
form a spanning tree.}
\end{figure}
This graph contains a Stallings subgraph on the path corresponding
to $xyx^{-1}y^{-1}$. This subgraph has as spanning tree the 3 green
edges on it, and we add another green edge to create a spanning tree
for the full graph. As generators for the fundamental group we first
take $xyx^{-1}y^{-1}$ for the $x$ edge which is not in green, and
the second generator is $xy^{2}x$ for the $y$-labeled edge which
is not in the spanning tree. Using \claimref{injective_is_direct_summand}
we conclude that $\left\langle \left[x,y\right]\right\rangle $ is
a free factor of $\left\langle \left[x,y\right],xy^{2}x\right\rangle $.
Note that a priori, the group $\left\langle \left[x,y\right],xy^{2}x\right\rangle $
might be generated by 1 element, and then $\left\langle \left[x,y\right]\right\rangle $
is a free factor exactly if it equals the full group. The graph visualization
tells us that the second group is actually bigger and needs at least
two generators.
\end{example}

\newpage{}

If $H\leq N\leq F_{n}$ then we can construct the two Stallings graphs
$\Gamma_{H},\Gamma_{N}$ and then there is a natural labeled morphism
$\Gamma_{H}\to\Gamma_{N}$. If this morphism is injective, then by
the claim above we know that $H\leq_{*}N$ is a free factor of $N$.
However, the injectiveness of this map depends also on our initial
choice of basis for $F_{n}$, and in general $H$ can be a free factor
of $N$ even when the corresponding graph morphism is not injective.
The next result uses the Stallings graphs to show how to naturally
find two subgroups where the Stallings graphs are injective, and therefore
one is a free factor of the other. As we shall see later, this type
of result is exactly what we need in our reduction in \subsecref{The-reduction-idea}.
\begin{thm}
\label{thm:direct_summand_construction}Let $N_{j}\;,\;j\in J$ be
a directed system of subgroups of $F_{n}$, and let $S\subseteq N=\bigcap_{j\in J}N_{j}$.
Then there is some $S\subseteq H\leq N$ and $j_{0}$ such that $H$
is a free factor of any subgroup $N'$ such that $H\leq N'\leq N_{j_{0}}$.
If $S$ is finite, then we may take $H$ which is finitely generated.
\end{thm}

\begin{proof}
Consider the map $\Gamma_{S}\to\Gamma_{N}$. The image of this map
is a sub Stallings graph of $\Gamma_{N}$ corresponding to some $\Gamma_{H}$
for some $H\leq N$. Note that if $S$ is finite, than so is $\Gamma_{S}$
and $\Gamma_{H}$ and therefore $H$ is finitely generated. We claim
that there is some $j_{0}$ such that $\Gamma_{H}\to\Gamma_{N_{j_{0}}}$
is injective. Once we know this, if $H\leq N'\leq N_{j_{0}}$ is any
intermediate group, then $\Gamma_{H}\to\Gamma_{N_{j_{0}}}$ is the
composition of $\Gamma_{H}\to\Gamma_{N'}\to\Gamma_{N_{j_{0}}}$ making
$\Gamma_{H}\to\Gamma_{N'}$ injective as well. By \claimref{injective_is_direct_summand}
it follows that $H\leq_{*}N'$ for any such $N'$, thus completing
the proof.

To find such a $j_{0}$, for each $v\in\Gamma_{H}$ let $P_{v}$ be
a path from $v_{0}$ to $v$ in $\Gamma_{H}$ where $v_{0}$ is the
special vertex. If $v_{1}\neq v_{2}$ in $\Gamma_{H}$, then $g_{v_{1},v_{2}}:=\pi_{1}\left(P_{v_{1}}\right)\pi_{1}\left(P_{v_{2}}\right)^{-1}\notin H$
and because $\Gamma_{H}\subseteq\Gamma_{N}$, this element is not
in $N$ as well. Hence, we can find some $j_{v_{1},v_{2}}$ such that
$g_{v_{1},v_{2}}\notin N_{j_{v_{1},v_{2}}}$, and using the directedness
of $N_{j}$, we can find $j_{0}$ such that $g_{v,u}\notin N_{j_{0}}$
for any two distinct vertices $v,u\in\Gamma_{H}$. This implies in
turn that the map $\Gamma_{H}\to\Gamma_{N_{j_{0}}}$ is injective
on the vertices, and since this is a morphism of Stalling graphs it
is injective on the edges as well, which is exactly what we needed.
\end{proof}
\begin{example}
Consider the set $S=\left\{ x^{3},y^{2}\right\} $ where $N=\left\langle y,xyx^{-1},x^{3}\right\rangle $.
In this case the image of $S$ in $\Gamma_{N}$, as can be seen in
the figure below, is the group $H=\left\langle x^{3},y\right\rangle $.
The element $x^{3}$ is mapped to $x^{3}$ while $y^{2}$ circles
twice around $y$.
\end{example}

\begin{figure}[H]
\begin{centering}
\includegraphics[scale=0.6]{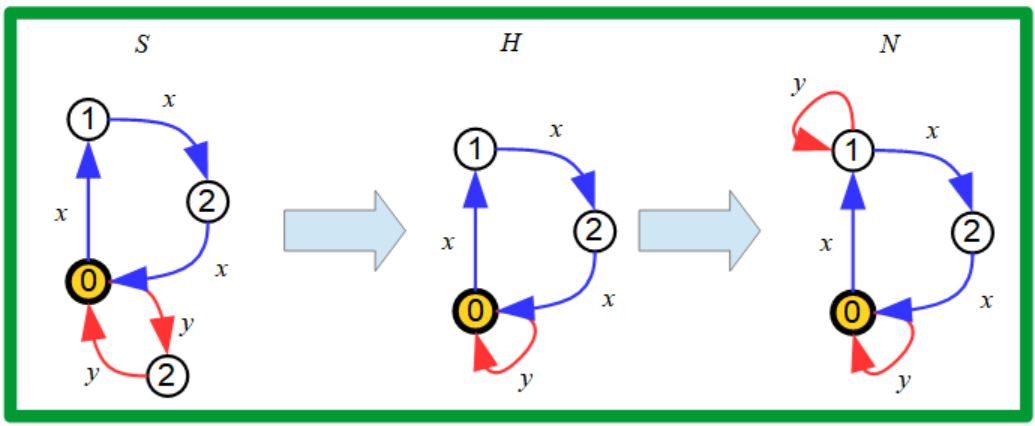}
\par\end{centering}
\caption{\label{fig:factor}The set $S=\left\{ x^{3},y^{2}\right\} $ is contained
in the (finitely generated) $H=\left\langle x^{3},y\right\rangle $
which is a free factor of $N=\left\langle y,x^{3},xyx^{-1}\right\rangle $.}
\end{figure}

\subsection{\label{subsec:Proof-main}Proof of the main theorem}

Let $g\in F_{k}$ and $m\in\ZZ$. Our profinite notation shows that
$x^{m}=g$ has the local to global property if a solution over $\hat{F}_{k}$
implies a solution over $F_{k}$. This is a very special type of equation
which can be written as $w=g$, where the left side contains only
parameters and the right side contain only $g$. In this section we
provide the details for the ideas in \subsecref{The-reduction-idea}
for this type of equations and then the specialization for $x^{m}=g$
will lead to the proof of \thmref{main_theorem}.

For the rest of this section, we will have two free groups. The first
will be denoted by $F_{n}$ and the free variables will be from it,
while the second $G=F_{k}$ will be the group in which we try to prove
the local global principle. 
\begin{defn}
Let $w\in F_{n}$. We say that $w$ satisfies the local to global
property in $F_{k}$, if for any $g\in G=F_{k}$, the equation $w=g$
is solvable in $\hat{G}$ if and only if it is solvable in $G$. We
say that $w$ satisfies the local global property in free groups,
if it satisfies it for every $k\geq1$.
\end{defn}

\begin{thm}
\label{thm:reduction_step}Let $w\in F_{n}$. Then $w$ satisfies
the local to global property in free groups, if and only if it satisfies
it in $F_{k}$ for any $k\leq n$.
\end{thm}

\begin{proof}
The $\Rightarrow$ direction is clear. Let us assume that $w$ satisfies
the local to global property for $k\leq n$ and show that it holds
for general $k$.

Let $k\in\NN$, $g\in F_{k}$, and assume that $w=g$ is solvable
in $\hat{F}_{k}$. Let $\hat{g}_{1},...,\hat{g}_{n}\in\hat{F_{k}}$
such that $w\left(\hat{g}_{1},...,\hat{g}_{n}\right)=g$ and set $H=\overline{\left\langle \hat{g}_{1},...,\hat{g}_{n}\right\rangle }$.
Since $H$ is closed in $\hat{F}_{k}$, it is the intersection of
all the closed finite index subgroups of $\hat{F}_{k}$ which contain
it. However, by \lemref{finite_index_correspondence} these subgroups
are in bijection with the finite index subgroups of $F_{k}$ via the
maps $N\mapsto\overline{N}$ and $\overline{N}\mapsto F_{k}\cap\overline{N}$.
Thus if we denote this set by $J=\left\{ N\;\mid\;N\underset{f.i.}{\leq}F_{k},\;H\leq\overline{N}\right\} $,
which is a directed set, then $H=\bigcap_{N\in J}\overline{N}$. More
over, since $g\in H\leq\overline{N}$ for every $N\in J$ we conclude
that $g\in\overline{N}\cap F_{k}=N$, and therefore $g\in\bigcap_{N\in J}N$.
Hence we have the following subgroups:
\[
\left\langle g\right\rangle \leq\bigcap_{N\in J}N\leq H=\bigcap_{N\in J}\overline{N}
\]

Applying \thmref{direct_summand_construction} for $S=\left\{ g\right\} $,
we can find $\left\langle g\right\rangle \leq H_{0}\leq\bigcap_{N\in J}N$
with $H_{0}$ finitely generated and $H_{0}\leq_{*}N$ for some $N\in J$,
so we may define a projection $\pi:N\to H_{0}$ which is the identity
on $H_{0}$. The group $N$ has finite index in $F_{k}$ and $H_{0}$
is a free factor in $N$, so by \lemref{induced_profinite} their
subspace topology is the profinite topology. Moreover, the group $N$
is finitely generated as a finite index subgroup of the finitely generated
group $F_{k}$, and of course $H_{0}$ is finitely generated by assumption.
We can now use \corref{extension_profinite} to find a continuous
extension $\hat{\pi}:\overline{N}\to\overline{H}_{0}$ which is the
identity on $\overline{H}_{0}$.

Since $g\in\overline{H}_{0}$, we get that 
\[
g=\pi\left(g\right)=\pi\left(w\left(\hat{g}_{1},...,\hat{g}_{n}\right)\right)=w\left(\pi\left(\hat{g}_{1}\right),...,\pi\left(\hat{g}_{n}\right)\right),
\]
so that $w=g$ is solvable in $\overline{H}_{0}$. Moreover, because
$\overline{H}_{0}\leq H\leq\overline{N}_{j}$, then the restriction
of $\pi$ to $H$ is also surjective on $\overline{H}_{0}$, implying
that $\overline{H}_{0}$ is generated topologically by $\pi\left(\hat{g}_{1}\right),...,\pi\left(\hat{g}_{n}\right)$.

The group $H_{0}$ is finitely generated subgroup of a free group,
and therefore $H_{0}\cong F_{k'}$ for some $k'$. Also, since by
\lemref{induced_profinite} its subspace topology is the profinite
topology we conclude that $\overline{H}_{0}\cong\hat{H}_{0}\cong\hat{F}_{k'}$.
But $\hat{F}_{k'}$ is generated topologically by $n$ elements, so
that $k'\leq n$. Indeed, if $x_{1},...,x_{k'}$ is a basis for $F_{k'}$
then there is the projection $\pi:F_{k'}\to\FF_{2}^{k'}$ sending
$x_{i}$ to the standard basis element $e_{i}$. Using \lemref{continuous_hom_extension},
we can extend $\pi$ to the projection $\pi:\hat{F}_{k'}\to\FF_{2}^{k'}$.
If $\hat{F}_{k'}$ is generated topologically by $n$ elements, then
so is any of its quotients, and since $\FF_{2}^{k'}$ cannot be generated
by less than $k'$ elements, we conclude that $k'\leq n$.

To summarize, the equation $w=g$ has a solution in $\hat{H}_{0}\cong\hat{F}_{k'}$
where $g\in H_{0}\cong F_{k}$. Under our assumption, the word $w$
has the local to global property for $k'\leq n$, so it has a solution
in $H_{0}\leq G$, thus completing the proof.
\end{proof}
\newpage{}

Finally, we can use this reduction to prove that $w=x^{m}$ has the
local to global property for free groups.
\begin{proof}[Proof of \thmref{main_theorem}]
 Since $w\in F_{1}$, by the reduction step in \thmref{reduction_step},
it is enough to prove that $w$ has the local to global property in
$F_{1}\cong\ZZ$. But we already saw that this is true, hence $w$
has the local to global property for free groups.
\end{proof}

\appendix

\section{\label{sec:Some-profinite-results}Some profinite results}

In this section we collect all sort of results about profinite groups
which are well known, but we add them here for the convenience of
the reader.

We start by understanding the continuous homomorphisms between profinite
groups and their profinite completions.
\begin{lem}
\label{lem:profinite_hom}Any homomorphism $\varphi:G_{1}\to G_{2}$
between groups with the profinite topologies is continuous.
\end{lem}

\begin{proof}
The profinite topology is defined as the weakest topology where all
the projections to finite groups are continuous. Hence, we need to
show that if $H$ is finite and $\pi:G_{2}\to H$ is a homomorphism,
then $\varphi\circ\pi:G_{1}\to G_{2}\to K$ is continuous. But this
is true by definition of the profinite topology on $G_{1}$, which
completes the proof.
\end{proof}
\begin{lem}
\label{lem:continuous_hom_extension}Let $G$ be a dense subgroup
of the metric group $\overline{G}$ and $H$ a compact metric group.
Then any continuous homomorphism $\varphi:G\to H$ has a unique continuous
extensions to a homomorphism $\overline{\varphi}:\overline{G}\to H$.
\end{lem}

\begin{proof}
Given $g\in\overline{G}$ we can find a sequence $g_{i}\in G$ such
that $g_{i}\to g$. Because $H$ is compact, by restricting to a subsequence
we may assume that $\varphi\left(g_{i}\right)\to h$ for some $h\in H$.
If $g_{i}'\to g$ is any other such sequence with $\varphi\left(g_{i}'\right)\to h'$,
then $g_{i}'g_{i}^{-1}\to e$ so that 
\[
h'h^{-1}=\limfi i{\infty}\varphi\left(g_{i}'\right)\varphi\left(g_{i}^{-1}\right)=\varphi\left(e\right)=e,
\]
so we see that $\overline{\varphi}:g\mapsto h$ is well defined. In
particular, for $g\in G$, we may take $g_{i}=g$, so that $\overline{\varphi}\left(g\right)=\varphi\left(g\right)$
, namely it is an extension of $\varphi$.

It is now a standard exercise to show that $\overline{\varphi}$ is
a continuous homomorphism.
\end{proof}
\begin{cor}
\label{cor:extension_profinite}Let $G,H$ be finitely generated,
residually finite groups with the profinite topology. Then any homomorphism
$\varphi:G\to H$ can be be extended uniquely to a continuous homomorphism
$\hat{\varphi}:\hat{G}\to\hat{H}$.
\end{cor}

\begin{proof}
By \lemref{profinite_hom} we know that $\varphi$ is continuous,
and therefore its composition with the embedding $H\hookrightarrow\hat{H}$
is continuous. To apply \lemref{continuous_hom_extension} we first
note that by definition $G$ is dense in $\hat{G}$, and every profinite
completion, and in particular $\hat{H}$ is compact. Finally, if a
group $G$ is finitely generated, then it has only countably many
finite quotients. It follows that $\prod\nicefrac{G}{K}$ is a countable
product of (discrete) metric spaces and therefore it is a metric space
in itself. In our case, we get that $G,H,\hat{G}$ and $\hat{H}$
are metric space. We can now apply \lemref{continuous_hom_extension}
to prove this lemma.
\end{proof}
In our proofs we work with subgroups of a group with the profinite
topology, so we want a simple condition when such a subgroup inherits
the profinite topology as the subspace topology. Once a subgroup $H\leq G$
has the profinite topology, we also want to show that its closure
$\overline{H}$ in $\hat{G}$ is isomorphic to $\hat{H}$.
\begin{lem}
\label{lem:induced_profinite}Let $G$ be a group with the profinite
topology and $H\leq G$ a subgroup. Then the induced topology on $H$
is the profinite topology if one of the following is true.
\begin{enumerate}
\item $H$ is a free factor of $G$.
\item $H$ is a finite index subgroup of $G$.
\item $H$ can be reached by a finite sequence of taking free factors and
finite index subgroups.
\end{enumerate}
\end{lem}

\begin{proof}
In general, the induced topology on $H$ is the weakest topology such
that any composition $H\to G\to K$, $K$ finite, is continuous. It
follows that this topology is weaker than the profinite topology on
$H$. To show equality we need to show that any $\varphi:H\to K$,
$K$ finite, is continuous, or equivalently the kernel is open in
$H$.
\begin{enumerate}
\item Write $G=H*\tilde{H}$ and let $\pi:G\to H$ be the projection which
fixes $H$ and sends $\tilde{H}$ to the identity. If $K$ is finite,
and $\varphi:H\to K$, then $G\overset{\pi}{\to}H\overset{\varphi}{\to}K$
is continuous by the definition of the profinite topology on $G$,
and since $H\hookrightarrow G$ is continuous, then so is $\varphi:H\to G\overset{\pi}{\to}H\overset{\varphi}{\to}K$,
which completes this case.
\item Let $\varphi:H\to K$ where $K$ is finite and set $N=\ker\left(\varphi\right)$
. Under the assumption that $\left[G:H\right]<\infty$, we get that
$\left[G:N\right]<\infty$. It follows that $N$ is closed and open
in $G$ and therefore in $H$, implying that $\varphi$ is continuous.
\item Follows by induction.
\end{enumerate}
\begin{lem}
Let $G$ be a group with the profinite topology and $H\leq G$ a subgroup,
both of which are finitely generated. If the induced topology on $H$
is the profinite topology, then $\overline{H}$ in $\hat{G}$ is naturally
isomorphic to $\hat{H}$.
\end{lem}

\begin{proof}
Consider the continuous map $\varphi:H\to\overline{H}\leq\hat{G}$.
Since $\hat{G}$ is compact, then so is $\overline{H}$, so we can
then use \lemref{continuous_hom_extension} to extend $\varphi$ to
a continuous map $\hat{\varphi}:\hat{H}\to\overline{H}$. Since $\hat{\varphi}$
is injective on the dense subgroup $H$ inside $\hat{H}$ (both of
which are metric spaces) it must also be injective on $\hat{H}$.
On the other hand, since $H$ is dense in $\overline{H}$ and $H\leq Im\left(\hat{\varphi}\right)$,
we conclude that $\hat{\varphi}$ is surjective as well. Finally,
since $\hat{\varphi}$ is a continuous bijection between compact and
Hausdorff spaces, its inverse is continuous as well, so we conclude
that $\hat{\varphi}:\hat{H}\to\overline{H}$ is a homeomorphism as
well, thus completing the proof.
\end{proof}
Finally, the next lemma shows how to understand the topology on $\hat{G}$.
This topology is generated by open (and closed) finite index subgroups
of $\hat{G}$ which correspond to finite index subgroups of $G$.
\end{proof}
\begin{lem}
\label{lem:finite_index_correspondence}Let $G$ be a finitely generated
group with the profinite topology. The map $N\to\overline{N}\leq\hat{G}$
is a bijection between finite index subgroups of $G$ and finite index
open and closed subgroups of $\hat{G}$ with the inverse map $\overline{N}\mapsto\overline{N}\cap G=N$. 
\end{lem}

\begin{proof}
The trick here if $\varphi:\hat{G}\to H$ is continuous for some finite
group $H$ with the discrete topology, then the inverse of any subset
from $H$ is closed an open. But if $U\subseteq\hat{G}$ is such a
set, then since $G$ is dense in $\hat{G}$, for any $\hat{g}\in U$
we can find $g_{i}\in U\cap G$ which converge to $\hat{g}$. It follows
that $U\subseteq\overline{U\cap G}\subseteq\overline{U}=U$ so we
get an equality $U=\overline{U\cap G}$. In particular this is true
for finite index closed and open subgroups of $\hat{G}$.

For the other direction, if $K\fidealeq G$, then the map $\pi:G\to\nicefrac{G}{K}$
is continuous in the profinite topology (by definition), so by \lemref{continuous_hom_extension}
we have the continuous extension $\hat{\pi}:\hat{G}\to\nicefrac{G}{K}$.
In particular $\hat{\pi}^{-1}\left(e\right)\fidealeq\hat{G}$ is a
closed and open subgroup such that $\hat{\pi}^{-1}\left(e\right)\cap G=\pi^{-1}\left(e\right)=K$,
so by our argument above $\hat{\pi}^{-1}\left(e\right)=\overline{K}$. 
\end{proof}
\bibliographystyle{plain}
\bibliography{algebra}

\begin{thebibliography}{1}

\bibitem{ceccherini-silberstein_cellular_2010}
Tullio Ceccherini-Silberstein and Michel Coornaert.
\newblock {\em Cellular {Automata} and {Groups}}.
\newblock Springer {Monographs} in {Mathematics}. Springer-Verlag, Berlin
  Heidelberg, 2010.

\bibitem{khelif_finite_2004}
Anatole Khelif.
\newblock Finite approximation and commutators in free groups.
\newblock {\em Journal of Algebra}, 281(2):407--412, November 2004.

\bibitem{ribes_profinite_2010}
Luis Ribes and Pavel Zalesskii.
\newblock {\em Profinite {Groups}}.
\newblock Springer Science \& Business Media, March 2010.
\newblock Google-Books-ID: u8GWrhdhA2QC.

\bibitem{thompson_power_1997}
John~G. Thompson.
\newblock Power {Maps} and {Completions} of {Free} {Groups} and of the
  {Modular} {Group}.
\newblock {\em Journal of Algebra}, 191(1):252--264, May 1997.

\bibitem{wilson_profinite_1998}
John~S Wilson.
\newblock {\em Profinite groups}.
\newblock Clarendon Press ; Oxford University Press, Oxford; New York, 1998.
\newblock OCLC: 40658188.

\end{thebibliography}

\end{document}